\documentclass[12pt,twoside,draft,reqno]{amsart}

\usepackage{amssymb}

\newtheorem{theorem}{Theorem}[section]
\newtheorem{lemma}[theorem]{Lemma}
\newtheorem{corollary}[theorem]{Corollary}

\newtheorem{remark}[theorem]{Remark}

\def\to{\rightarrow}

\def\bs{\bigskip}

\def\AA{{\mathcal{A}}}
\def\PP{{\mathcal{P}}}

\def\CC{{\mathcal{C}}}

\def\NN{{\mathcal N}}

\title{Pseudocontinuation and Cyclicity \\ for Random Power Series}

\author{Evgeny~Abakumov}
\address{Laboratoire d'Analyse et de Mathematiques Appliqu\'ees, UMR CNRS 8050
\\Universit\'e de Marne-la-Vall\'ee
\\5, Bd Descartes
\\Champs-sur-Marne
\\77454 Marne-la-Vall\'ee Cedex 2
\\France}
\email{abakumov\@@math.univ-mlv.fr}
\author{Alexei~Poltoratski}
\thanks{The second author is supported by
N.S.F. Grant No. 0500852}
\address{Texas A\&M University
\\ Department of Mathematics\\
College Station, TX 77843, USA }
\email{alexeip\@@math.tamu.edu}

\begin{document}
\maketitle

\begin{abstract}
We prove that a random function in the Hardy space $H^2$ is a non-cyclic vector for the backward
shift operator almost surely. The question of existence of a local pseudocontinuation for
a random analytic function is also studied.
\end{abstract}

\section{Introduction}\label{sec1}

The study of various properties of random analytic functions has become an integral
part of complex analysis during the last century.
A natural way to introduce "randomness" on a space of analytic functions is to
consider
random power series, i. e. power series whose coefficients are independent
complex-valued random variables.
This approach was first utilized by Borel \cite{Bo} at the end of 19th century.
Borel was concerned with the absence of  analytic
continuation for a random function.
Later such studies were continued, among others, by
Steinhaus \cite{St}, Payley and Zygmund \cite{PZ}.
It was shown that random analytic functions with Steinhaus- or
Rademacher-distributed coefficients cannot
be analytically continued through any arc of their circle of convergence almost
surely.
The book by J.-P. Kahane \cite{K}
contains an excellent introduction to this circle of questions as well as further
results and historical references.

The notion of pseudocontinuation (see Section \ref{sec2} for the definition), that
generalizes
the notion of analytic continuation, was introduced by H. Shapiro in \cite{Sh}.
The important role of  pseudocontinuations in operator theory was discovered by
Douglas, Shapiro and Shields in \cite{DSS}. In addition to the operator theory,
pseudocontinuable functions appear in
a number of applications in function theory, in the theory of stationary Gaussian
processes, etc., see \cite{RS, N}
or \cite{R} for a recent review of this area.
One of the goals of this paper is to extend the classical results on the absence of
analytic continuations
for random analytic functions mentioned above to the case of pseudocontinuations.

The main result of \cite{DSS} establishes a relation between pseudocontinuation and cyclicity.
Recall that a vector $v$ is cyclic for a linear operator $L$ on a Banach space if
the closed linear span of
$L^n v, n=1,2,...$ is equal to the
whole space.
We will denote by $B$ the standard operator of backward shift on the Hardy
space $H^2$ (see, for instance, \cite{N, CR}) in
the unit disk $\Bbb D$ defined as
$$Bf(z)=\frac{f(z)-f(0)}z.$$
Since $B$ is the only linear operator considered in this paper, we
will call a vector from $H^2$ cyclic if it is cyclic for $B$. The
backward shift operator is one of the main objects of the functional
model theory. For a thorough discussion of this theory we refer the reader to the well-known
book by Nikolski \cite{N} . It was shown in
\cite{DSS} that a function from $H^2$ is a cyclic vector for $B$ if and only if
it does not admit a pseudocontinuation into the complement of the
unit disk $\hat{\Bbb C}\setminus\bar{\Bbb D}$. This connection between (non-)cyclicity
and pseudocontinuation allows one to establish numerous properties
of the set of cyclic vectors $\CC$.

One of such properties of $\CC$  is that it is a dense
$G_\delta$-subset of $H^2$ with respect to the norm topology. I.e.,
quasi almost all functions in $H^2$, in the sense of Baire, are
cyclic. In view of this fact it is natural to ask  if a "randomly
chosen"  vector from $H^2$ is cyclic almost surely. This question
was
 brought to our attention by
N.~Nikolski and, independently, by D.~Sarason. Here we give a positive answer to
this question in the form of
the following theorem:

\begin{theorem}\label{t1}
Consider the random power series
$$F(z)=\sum_{n=0}^\infty X_nz^n,\ \ \ X_n=X_n(w),$$
where the coefficients $X_n$ are independent random complex variables. Suppose that
$F(z)$ belongs to $H^2$ almost surely.
Then either

(i) $F(z)$ is a cyclic vector almost surely, or

(ii) there exists a (deterministic) function $F_0\in H^2$ such that $F-F_0$ is a polynomial
almost surely.
\end{theorem}

\begin{remark}
Classical situations of Steinhaus, Gaussian or Rademacher random variables are
covered by (i).
Case (ii) is pathological: it occures when the random variables $X_n$ asymptotically
 resemble  point charges.
In this case $F$ is cyclic a.s. if $F_0\in \CC$ and non cyclic a.s.
if $F_0\in H^2\setminus \CC$.
\end{remark}

Theorem \ref{t1} is proved in Section \ref{sec2}.
The proof is based on non-pseudocontinuability of  lacunary series studied in \cite{ DSS, Ab, Al1, Al2}.
In our argument we make use of the remarkable results by A. Aleksandrov from \cite{Al2}.

Section \ref{sec3} is devoted to local pseudocontinuations. The
results from \cite{St}, \cite{PZ}, and \cite{RN} show that (in most
cases) random power series do not admit local analytic continuations
almost surely. It is natural to ask if these results will remain
true after "continuation" is replaced with "pseudocontinuation".
 A positive answer to this question is given in  Theorems \ref{t3}
and \ref{t4}.



\section{Pseudocontinuation into $\hat{\Bbb C}\setminus\bar{\Bbb D}$ and cyclicity}\label{sec2}

In this paper we are using the standard probabilistic setup described
in \cite{K}, where the reader can find all the necessary definitions and basic
properties.

Let $(\Omega, \AA, \PP)$ be a probability space. We will consider a
sequence of independent random complex variables $X_n\,\,
(X_n=X_n(w), w\in \Omega)$ and study the properties of the random
power series $F(z)=\sum_{n=0}^\infty X_n z^n$. When studying such
properties one should start by verifying that a property is an
event, i.e. that the set $A$ of $w$, such that the function
corresponding to the sequence $\{X_n(w)\}$ satisfies the desired
property, is measurable: $A\in\AA$. We will be mostly interested in
the events $A$ that do not depend on the values of any finite number
of variables $X_n$, i. e. the sets $A\in\AA$ with the property that
if $w\in A$ and $X_n(w)=X_n(w')$ for all but finitely many $n$ then
$w'\in A$. By the zero-one law the probability of any such event is
0 or 1.

First let us list some examples of such events $A$:

1) The radius of convergence of $F$ is equal to $r$, where $r$ is a
fixed positive number or $ +\infty$;

2) $F(z)$ belongs to the space $H^2$;

3) $F(z)$ is a cyclic vector in $H^2$.

In some cases it is clearly true that the property is an event: in
the first example it suffices to see that $r(w)=(\limsup
X_n^{1/n})^{-1}$ is a measurable function of $w$. In other cases
this step requires a considerable effort, like in our 3rd example
(it will follow from Lemma \ref{l2}). After the measurability is
established, it is usually clear if an event like 1)-3) depends on
the values of a finite number of variables. If it does not, one can
apply the zero-one law.

In our first example the zero-one law implies that the random series $F(z)$ has the same
radius of convergence for almost all $w$. Throughout this paper we will denote this
deterministic radius by $r_F$.

Let $f$ be analytic in the disk $\{|z-a|<r\}$ and let $I$ be an arc
of the boundary circle $\{|z-a|=r\}$. We say that $f$ admits a
pseudocontinuation through $I$ into an open (in
$\hat{\Bbb C} $) domain $D\subset\{|z-a|>r\},
\partial D\supset I$, such that $\{|z-a|<r\}\cup I\cup D$ is open,
if there exists a  function $g$ from the Nevanlinna class  $\NN(D)$ such that
the non-tangential limits of $f$ and $g$ coincide almost everywhere
on $I$. The function $f$ is said to be
pseudocontinuable through $I$ if it admits a pseudocontinuation into some domain $D$
as above. Recall that the Nevanlinna
class $\NN(D)$ is defined as a set of ratios of two bounded analytic
functions in $D$.


Now we will formulate the following result  on pseudocontinuation of a random power
series. It will imply the cyclicity result
Theorem \ref{t1}.

\begin{theorem}\label{t2}
Consider a series $F(z)=\sum_{n=0}^\infty X_n z^n$, where $X_n=X_n(w)$ are
independent random complex variables. Suppose that the radius
of convergence of the series $F(z)$ is equal to $r>0$ almost surely. For any $n\geq 0$
denote $\epsilon_n=1-\max_{\zeta\in \Bbb C} P(X_n=\zeta)$.

a) If $\sum_{n=0}^\infty\epsilon_n=\infty$ then $F(z)$ does not
admit a pseudocontinuation through its circle of convergence 
into $\hat{\Bbb C}\setminus\overline{r\Bbb D}$ almost surely.

b) If $\sum_{n=0}^\infty\epsilon_n<\infty$ then there exists a
(deterministic) power series $F_0(z)=\sum_{n=0}^\infty \hat{F_0}(n)
z^n$ such that $F-F_0$ is a polynomial almost surely.
\end{theorem}

An important role in the proof of Theorem \ref{t2} is played by a result on $R$-sets
from \cite{Al2}.
Following \cite{Al2} we say that a subset $\Lambda$ of $\Bbb Z_+$ is an $R$-set if
any analytic in the unit disk
function $f(z)=\sum_{n\in \Lambda}a_n z^n$,
that has a pseudocontinuation across $\Bbb T$, 
also admits an analytic continuation across $\Bbb T$.

\begin{theorem}\label{t6}\cite{Al2}
Let $\Lambda\subset\Bbb Z_+$. Suppose that for all sufficiently
large positive integers $s$  there exists a positive integer $m$
such that
$$\Lambda\cap\{-s+m\Bbb Z\}=\emptyset.$$ Then $\Lambda$ is an $R$-set.
\end{theorem}

A trivial example of an $R$-set is a finite set. Among the non-trivial examples, let
us mention the set
of all squares $\{n^2\}_{n\in\Bbb Z}$ and the set of all prime numbers, see
\cite{Al2}.

To prove Theorem \ref{t2} we will need the following two lemmas. The first one
concerns  arithmetic properties of $R$-sets, and the second is the statement on the
measurability of the considered sets
of power series.

\begin{lemma}\label{l1}
Let $\sum_{n=0}^\infty\epsilon_n$ be a divergent series of
non-negative real numbers. Then there exists an $R$-set
$\Lambda\subset\Bbb Z_+$ such that $\sum_{n\in \Lambda}\epsilon_n$
diverges.
\end{lemma}

\begin{proof}

We start by proving that if for two positive integers $s_1$ and $s_2$

$$\sum_{n\not\in\{-s_i+m\Bbb Z\}}\epsilon_n<\infty$$
for all $m\geq 1$, then $s_1=s_2$.

Indeed, assume that $s_1\neq s_2$ and choose $m>\max(s_1,s_2)$. Then
the complements (in $\Bbb Z_+$) of the sets $\{-s_i+m\Bbb Z\}, \,
i=1,2$, \,cover $\Bbb Z_+$ and hence
$$\sum_{n\not\in\{-s_i+m\Bbb Z\}}\epsilon_n=\infty$$
either for $i=1$ or for $i=2$.

Therefore there is at most one $s_0>0$ such that for any $m\geq 1$
$$\sum_{n\not\in\{-s_0+m\Bbb Z\}}\epsilon_n<\infty.$$
If such an $s_0$ does not exist, put $s_0=0$.
It follows that for any
$s>s_0$ there is $m>0$ such that
$$\sum_{n\not\in\{-s+m\Bbb Z\}}\epsilon_n=\infty.$$

Now we construct $\Lambda$ using a variation of the diagonal process.

Choose $m_1>s_0+1$ such that
$$\sum_{n\not\in\{-s_0-1+m_1\Bbb Z_+\}}\epsilon_n=\infty\ \text{ and set }
\Lambda_1={\Bbb Z}_+ {\setminus} \{- s_0 - 1 + m_1 \Bbb Z \}.
$$
Then choose $m_2>s_0+2$ sufficiently large so that
$$\sum_{n\not\in\{-s_0-2+m_2\Bbb Z\}}\epsilon_n=\infty
\ \ \text{ and }
\sum_{n\in\Lambda_1\cap [0,-s_0-2+m_2]}\epsilon_n>1.$$
Set $\Lambda_2=\Lambda_1\setminus \{-s_0-2+m_2\Bbb Z_+\}$.

On the $k$-th step choose $m_k>s_0+k$ sufficiently large so that
$$\sum_{n\not\in\{-s_0-k+m_k\Bbb Z\}}\epsilon_n=\infty\ \ \text{ and }
\sum_{n\in\Lambda_{k-1}\cap [0,-s_0-k+m_k]}\epsilon_n>k-1.$$
Set $\Lambda_k=\Lambda_{k-1}\setminus \{-s_0-k+m_k\Bbb Z_+\}$.

Now put $\Lambda=\cap_{k=1}^\infty \Lambda_k$. By our construction
$\sum_{n\in \Lambda} \epsilon_n =\infty$ and $\Lambda$ satisfies
the hypothesis of Theorem \ref{t6}. Thus $\Lambda$ is an $R$-set.
\end{proof}

\begin{lemma}\label{l2} Let $F$ be as in the statement of Theorem \ref{t2} and let
$I$ be an arc
of the circle of convergence of $F$.
The set of all $w$  corresponding to funcitons $F$  that admit a pseudocontinuation
through $I$ is an event.
\end{lemma}

\begin{proof}
It is enough to show that the set of $w$ defined by $F$ that admit a
pseudocontinuation through  $I$ into a fixed domain $D$, adjacent to
$I$, is an event.

Consider the Nevanlinna class $\NN (D)$. Without loss of generality one can assume that $D$ is a simply
connected domain with smooth boundary. Let $\phi$ be a conformal map from $D$ to $\Bbb D$.
Any function $f$ from $\NN (D)$ can be represented as $g\circ\phi$,
where $g$ is from the Nevanlinna class in $\Bbb D$. The function $g$ can be represented as
\begin{equation} g=\exp\left(\int_{\Bbb T}\frac{1+\bar\xi z}{1-\bar\xi z}d\mu(\xi)\right)\frac {B_1}{B_2}
\label{e1}\end{equation}
where $\mu$ is a finite real measure on $\Bbb T$ and $B_1,B_2$ are Blaschke products.
Let us denote by $S_1$ and $S_2$ the Blacshke sums for $B_1$ and $B_2$ correspondingly. We will
denote by
$\NN_n$ the subset of $\NN (D)$ consisting of functions $f$, for which the corresponding $\mu, B_1$ and $B_2$
satisfy $||\mu||+S_1+S_2\leq n$.

Denote by $L^0(I)$ the set of all Lebesgue-measurable functions on $I$ with the
topology of convergence in measure.
Denote by $L_n$ the subset of $L^0(I)$ consisting of the boundary values of
functions from $\NN_n$.
Let $C_n$ be a countable subset of $\NN_n$ that is dense in $L_n$ with respect
to the topology of $L^0(I)$. For $p\in C_n, m\in \Bbb N$, denote by $S_p^m$ the set
of all power series $F$ in $\Bbb D$
that satisfy
$$|\{|F-p|>1/m\}\cap I|<1/m.$$
The set of $w$ corresponding to  $S_p^m$ is a measurable set (an event). The set of
all functions $F$ admitting a pseudocontinuation from $\NN_n$
is the set
\begin{equation}\cap_m\cup_{p\in C_n}S_p^m,\label{e2}\end{equation}
which is, therefore, also an event. 

Indeed, it is easy to see that
the set in \eqref{e2} contains all functions $F$ admitting a
pseudocontinuation from $\NN_n$. To prove the other inclusion,
notice that functions $F$ contained in this set have the property
that in every neighborhood of $F$, with respect to the topology of
$L^0(I)$, there is a function $f\in \NN_n$. Hence one can choose a
sequence $f_k$ of such functions converging to $F$ a.e. on $I$.
Since all $f_k$ belong to $\NN_n$, from this sequence one can choose
a subsequence $f_{n_k}$ such that the measures $\mu_{n_k}$
from \eqref{e1} converge $*$-weakly on $\Bbb T$ and the  Blaschke products $B_1^{n_k}$ and $B_2^{n_k}$ converge normally in $\Bbb D$.

In other words $f_{n_k}=\varphi_k/\psi_k$, where $\varphi_k$ and $\psi_k$
are bounded (uniformly in $k$) analytic functions, converging normally
to $\varphi$ and $\psi$ correspondingly ($\psi\neq 0$ a.e. on $I$ because $f_{n_k}\in \NN_n$).
Note that normal convergence in $\Bbb D$ implies weak convergence
of the boundary values of $\varphi_k$ and $\psi_k$ in $L^2(I)$
to the boundary values of $\varphi$ and $\psi$ correspondingly.
One can show that together with the fact that $\varphi_k/\psi_k$
tends to $F$ a.e. on $I$, this implies that $F=\varphi/\psi$
a.e. on $I$.

Finally, the set of all $F$ admitting a pseudocontinuation through $I$ is an event
because it is equal to
$$\cup_n\cap_m\cup_{p\in C_n}S_p^m.$$
\end{proof}

Obviously, on any circle there exists a countable family of arcs
such that any arc contains an arc from that family. Together with
the last lemma this observation gives the following corollary, which
will be used in the next section.

\begin{corollary}\label{c1} Let $F$ be as in the statement of Theorem \ref{t2}.
The set of all $w$  corresponding to functions $F$ for which there
exists an arc $I$ of the circle of convergence, such that $F$ admits
a pseudocontinuation through $I$, is an event.
\end{corollary}

\begin{proof} [Proof of Theorem \ref{t2}.]
We will suppose without loss of generality that $r=1$, and we will say for simplicity that
an analytic function $f$ in the unit disk is pseudocontinuable if it
is pseudocontinuable through the unit circle into $\hat{\Bbb
C}\setminus\bar{\Bbb D}$. It follows from Lemma \ref{l2} that the set $A$ of all
pseudocontinuable functions is an event.

First, suppose that
$\sum_{n=0}^\infty \epsilon_n<\infty$. For every $n\geq 0$ denote by $a_n$  a
complex number
such that $P(X_n=a_n)= 1-\epsilon_n$ and set $F_0(z)=\sum_{n=0}^\infty a_n z^n$ (such
a number exists by the definition of $\epsilon_n$). Then the
probabilities
$$p_N=P(\hat{F}(n)=a_n\text{ for all $n\geq N$ })=\prod_{n=N}^\infty(1-\epsilon_n)$$
 tend
to 1 as $N\to\infty$. Therefore, $F-F_0$ is a polynomial almost surely. In
particular, $F$ is pseudocontinuable
almost surely if and only if $F_0$ is pseudocontinuable.

Consider now the essential case when $\sum_{n=0}^\infty\epsilon_n=\infty$.
By Lemma \ref{l1} one can choose a subset $\Lambda$ of $\Bbb Z_+$ such that
$\Lambda$ is an $R$-set and
$\sum_{n\in \Lambda}\epsilon_n=\infty$. Construct a sequence $\{T_n\}_{n\geq 0}$ of
automorphisms (measure-preserving bijections) of the probability
space $\Omega$ such that $T_n$ is the identity map for $n\not\in \Lambda$ and
$P(X_n(w)\neq X_n(T_n(w)))\geq 2\epsilon_n$
for $n\in \Lambda$.

To give an idea of how to construct such $T_n$ we will assume that
$\Omega$ is the circle
$\Bbb T$ with Lebesgue measure. Recall that by our definition
$\epsilon_n=1-\max_{\xi\in\Bbb C} P(X_n=\xi)$.
Then the automorphism $T_n$ of $\Omega=\Bbb T$, satisfying $P(X_n(w)\neq
X_n(T_n(w)))\geq 2\epsilon_n$, can
be chosen as an appropriate rotation of $\Bbb T$.

To complete the proof let us notice that,
since any polynomial is pseudocontinuable,
the existence of pseudocontinuation does not depend on any finite number of
coefficients. So by the zero-one law
the probability of the event "$F(z)$ is  pseudocontinuable" is equal to 0 or 1. We
will assume that this probability
is 1 and obtain a contradiction. If $F$ is pseudocontinuable almost surely then so is
$$\tilde F=\sum_{n=0}^\infty X_n(T_n(w))z^n$$
because the random variables $X_n(w)$ and $X_n(T_n(w))$ are similar. Hence the
difference
$$F(z)-\tilde F(z)=\sum_{n\in\Lambda}  \left(X_n(w)-X_n(T_n(w))\right)z^n $$
is pseudocontinuable almost surely. Since $\Lambda$ is an $R$-set, $F(z)-\tilde
F(z)$ must have an analytic
continuation through $\Bbb T$ into a neighborhood of $\bar{\Bbb D}$ almost surely. A
function analytic in $\Bbb D$ that has a meromorphic pseudocontinuation
to  $\hat{\Bbb C}\setminus\bar{\Bbb D}$ and, at the same time, has an analytic
continuation through $\Bbb T$ must be a rational function with poles
outside of $\Bbb D$.

So the function $F(z)-\tilde F(z)$ is rational a.s., and its spectrum (the set of all $n$ for which
$n$-th Taylor coefficient is non-zero) is included in $\Lambda$.
Now it follows easily (see also \cite{Al2}) that $F(z)-\tilde F(z)$ is a polynomial almost surely.
By the Borel-Cantelli
Lemma this contradicts  the condition that
$\sum_{n\in\Lambda}\epsilon_n=\infty$.

\end{proof}

\begin{proof}[Proof of Theorem \ref{t1}]

Suppose first that $\sum \epsilon_n<\infty$. The second statement of the theorem can
be proved in the same way
as the second statement of Theorem \ref{t2}.

Now let $\sum \epsilon_n=\infty$. Since $F\in H^2$ a.s., $r_F\geq 1$. If $r_F=1$ then the statement
follows directly from Theorem
\ref{t2}.

Suppose that $1<r_F<\infty$ and $F$ is non-cyclic with probability
greater than zero. Then by the 0-1 law it is non-cyclic with
probability one, i.e. it has a pseudocontinuation $\tilde F$ through
$\Bbb T$ into $\hat{\Bbb C}\setminus\bar{\Bbb D}$ a.s. From the
uniqueness of pseudocontinuation, $\tilde F$ coincides with the
analytic continuation of $F$ through $\Bbb T$ in $\{1<|z|<r_F\}$.
The restriction of $\tilde F$ to $\{|z|>r_F\}$ is a
pseudocontinuation of $F$ through the circle of convergence into
$\hat{\Bbb C}\setminus\overline{r\Bbb D}$ which exists a.s. This
contradicts Theorem \ref{t2}.

Finally, assume that $r_F=\infty$. If $F$ has a pseudocontinuation
through $\Bbb T$ into $\hat{\Bbb C}\setminus\bar{\Bbb D}$ a.s. and,
at the same time, is an entire function, then it is a polynomial
a.s. By the Borel-Cantelli Lemma this can only occur if $\sum
\epsilon_n<\infty$ and we obtain a contradiction.
\end{proof}

\section{Local pseudocontinuations}\label{sec3}

As was mentioned in the introduction, classical theorems on local
analytic continuation of random power series from \cite{St},
\cite{PZ} and \cite{RN} can be strengthened to cover local
pseudocontinuations. This is the goal of this section.

In Section \ref{sec2} we  proved that the properties of random series considered 
below are events (see Corollary \ref{c1}) and so these matters do not need to be
addressed in the proofs here. In the proofs we combine the  ideas from the
proof
of the Ryll-Nardzewski theorem given in \cite{K} with Lemma \ref{l1}.

Recall that $X$ is said to be a symmetric random variable if  $X$ and $-X$ are equidistributed.

 In Theorem \ref{t3} we
prove the a.s. absence of local pseudocontinuation for the
symmetric random variables. This includes the classical case of
Gaussian, Steinhaus and Rademacher variables.

\begin{theorem}\label{t3}Let  $F(z)=\sum_{n=0}^\infty X_n(w) z^n$ be a power series
where $\{X_n(w)\}$ is a sequence of symmetric independent random variables. Let
$r_F$ be the radius
of convergence of $F(z)$ and suppose that $0<r_F<\infty$. Then   $F(z)$ does not
admit a pseudocontinuation across any arc of the circle $\{|z|=r_F\}$ almost
surely.
\end{theorem}

\begin{proof}

Let  $I$ be  an open arc of the circle $|z|=r=r_F$. We will  prove
that $F(z)$ does not admit a pseudocontinuation across $I$ almost surely. Assume the
opposite. Then by the zero-one law $F(z)$  admits a pseudocontinuation
across $I$ almost surely. To obtain a contradiction we will show that this implies
that $F(z)$ admits a pseudocontinuation across the whole circle
$\{|z|=r\}$ almost surely.

To do this first choose an integer $m>\frac{2\pi r}{|I|}$. For any $k=0,1,2,...,m-1$
define
$\epsilon_n^k=-1$ if $n\equiv k$ (mod m) and $\epsilon_n^k= 1$ if $n\not\equiv k$
(mod m). Put $F_k(z)=\sum_{n=0}^\infty \epsilon_n^k X_n z^n$.
Then the functions $F_k$ are pseudocontinuable across $I$ because the random
variables $X_n$ are symmetric. Hence the same is true for
$$F(z)-F_k(z)=\sum_{j=0}^\infty 2X_{k+jm} z^{k+jm}=z^k H_k(z),$$
where the function $H_k$ has the property $H_k(z)=H_k(e^{\frac{2\pi i}m}z)$ for any
$k$. It follows that $H_k$ is pseudocontinuable not only through $I$ but
through the whole circle $\{|z|=r\}$ almost surely and therefore
$$F(z)=\frac 12\sum_{k=0}^{m-1} z^k H_k(z)$$
admits  a pseudocontinuation across $\{|z|=r\}$ almost surely.

Now recall that $r=(\limsup |X_n|^{1/n})^{-1}$ almost surely. Hence there exists a
sequence $\delta_n\downarrow 0$ such that
$$\sum_{n\geq 0} P(|X_n|^{-1/n}<r+\delta_n)=\infty.$$
By Lemma \ref{l1} there exists an $R$-set $\Lambda\subset\Bbb Z_+$ such that
$$\sum_{n\in\Lambda} P(|X_n|^{-1/n}<r+\delta_n)=\infty.$$
Put $X_n'=X_n$ if $n\not\in \Lambda$ and $X_n'=-X_n$ if $n\in\Lambda$. Then the
random variables $X_n$ and $X_n'$ are similar for all $n$.
Since $F(z)$ admits a pseudocontinuation across $\{|z|=r\}$ almost surely, so does
$F'=\sum X'_n z^n$. Therefore the random series
$$F-F'=2\sum_{n\in \Lambda}X_n(z)z^n$$
admits a pseudocontinuation across $\{|z|=r\}$ almost surely. Since $\Lambda$ is an
$R$-set, by Theorem \ref{t6}, $F-F'$ admits
an analytic continuation to a larger disk $\{|z|<r+2\epsilon\}$ almost surely and thus
$$\sum_{n\in\Lambda} P(|X_n|^{-1/n}<r+\epsilon)<\infty$$
(Borel-Cantelli Lemma), which contradicts our choice of $\Lambda$.

\end{proof}

Now we consider the general case. The statement of Theorem \ref{t3}, without the
requirement of symmetry, no longer holds true as shown
by the simple example where $X_n$ is equal to $1-\frac 1{2^n}$ or to $1+\frac
1{2^n}$ with probabilities
$\frac12$. Then $r_F=1$ but $F$ a.s. has an analytic continuation  across any arc
not containing 1.

\begin{theorem}\label{t4}
Consider a power series $F(z)=\sum_{n=0}^\infty X_n(w) z^n$, where $\{X_n(w)\}$ is a
sequence of independent random variables. Let $r_F$ be the radius
of convergence of $F(z)$ and suppose that $0<r_F<\infty$. Then there are two
possibilities:

1)  $F(z)$ does not admit a pseudocontinuation across any arc of the circle
$\{|z|=r_F\}$ almost surely, or

2) there exists a (deterministic) function $F_0(z)=\sum_{n=0}^\infty a_n z^n$ such that
$r_{F-F_0}>r_F$ and,
if $r_{F-F_0}<\infty$, $F-F_0$ does not
admit a pseudocontinuation across any arc of $\{|z|=r_{F-F_0}\}$ almost surely.

\end{theorem}

\begin{proof}

Suppose part 1) of the statement of the Theorem does not
hold. Then by the zero-one law
there exists an open arc $I$  of $\{|z|=r\}$ such that $F$ admits a
pseudocontinuation across $I $ almost surely. Using
$\Omega \times \Omega$ as a  new probability space consider
$$G(z,w,w')=\sum_{n=0}^\infty (X_n(w)-X_n(w'))z^n$$
where $(w,w')\in \Omega \times \Omega$. Now the coefficients of the random series
$G$ are symmetric independent random variables
and $G$ is pseudocontinuable across $I$ almost surely. So, by Theorem \ref{t3},
$r_G>r_F$ and, if
$r_G<\infty$, then $G$ does not admit a pseudocontinuation through any
arc of $\{|z|=r_G\}$. Clearly one can choose a point
$w_0\in \Omega$ such that for almost all $w$, $r_G(w, w_0)=r_G$, and $G(r,w,w_0)$
does not admit a pseudocontinuation across $\{|z|=r_G\}$. Set
$a_n=X_n(w_0),n\geq 0,$ and denote  $F_0(z)=\sum_{n\geq 0} a_n z^n$. We have
$$ F-F_0=\sum_{n=0}^\infty(X_n(w)-a_n)z^n=G(z,w,w_0),$$
and the result follows from our choice of $w_0$.

\end{proof}

\bs

\end{document}